\documentclass[12pt]{amsart}

\usepackage{amsmath,amssymb, amscd, stmaryrd}
\usepackage{fullpage}
\usepackage{enumerate}
\usepackage{cite}
\usepackage{framed}
\usepackage[dvips,dvipdf]{graphicx}
\usepackage[usenames,dvipsnames]{color}
\usepackage{color}
\usepackage{tabularx}
\usepackage{array}
\usepackage{multirow}

\usepackage{amsfonts}
\usepackage{epstopdf}
\usepackage{subfigure,caption,float}
\usepackage{enumerate}
\usepackage{cleveref}

\newcommand{\ncom}{\newcommand}
\ncom{\beqn}{\begin{eqnarray*}}
	\ncom{\eeqn}{\end{eqnarray*}}

\newtheorem{theorem}{Theorem}[section]
\newtheorem{proposition}[theorem]{Proposition}

\newtheorem{lemma}[theorem]{Lemma}

\theoremstyle{definition}
\newtheorem{definition}[theorem]{Definition}
\theoremstyle{definition}

\theoremstyle{remark}
\newtheorem{remark}[theorem]{Remark}

\title{Characterizations of approximation properties defined by operator ideals in the predual of weighted Banach spaces of holomorphic functions}

\author[Deepika Baweja]{Deepika Baweja} 
\address{Deepika Baweja, Department of Mathematics, BITS Pilani, Hyderabad Campus, India-500078.}
\email{deepika@hyderabad.bits-pilani.ac.in}

\author[Manjul Gupta]{Manjul Gupta}
\address{Manjul Gupta, Department of Mathematics and Statistics, IIT Kanpur, India-208016.}
\email{manjul@iitk.ac.in}

\begin{document}
\maketitle{\title}
\begin{abstract}
	In this article, we show that the predual $\mathcal{G}_w(U)$ of the
	weighted space of holomorphic functions
	has the $\mathcal{I}$- approximation
	property if and only if $E$ has the $\mathcal{I}$- approximation
	property, where $\mathcal{I}$ is a suitably chosen operator ideal, and
	$\it{w}$ is a radial weight defined on a balanced open subset U of a Banach
	space $E$.
	
\end{abstract}


\section{Introduction}
Approximation property is one of the most fundamental  properties in
the theory of Banach spaces as it approximates the identity operator
by finite rank operators uniformly on compact subsets of the Banach
space. This notion appeared for the first time in Banach's book in
1932 and later a systematic study of this property along with its
variants was carried out by A. Grothendieck \cite{Grothendieck} in 1955, who showed
the importance of the approximation property in the structural study
of Banach spaces. After the appearance of the work of P. Enflo \cite{Enflo},
who constructed an example of a Banach space lacking the
approximation property, researches in approximation property gained
momentum. Several variants of this property have been introduced and
studied extensively; for instance, compact approximation property,
weakly compact approximation property etc.. Replacing the ideals of
finite rank/compact/weakly compact operators by an arbitrary
operator ideal $\mathcal{I}$, S. Berrios and G. Botelho \cite{Sonia} studied
the concept of $\mathcal{I}$-approximation property which means that
the identity operator on $E$ is uniformly approximated by a member
of $\mathcal{I}\equiv\mathcal{I}(E,E)$ on compact subsets of $E$. This yields a
unification of several variants of the approximation property. In
\cite{Sonia}, the authors studied the $\mathcal{I}$-approximation
property in spaces of holomorphic functions of bounded type, spaces
of weakly uniformly continuous holomorphic functions, spaces of
bounded holomorphic functions and thus generalized some of the
results obtained in \cite{Caliskan, Mujica}.  Approximation properties 
more general than $\mathcal{I}$- approximation Property  defined corresponding to convex subsets of class of bounded linear operators have also
been investigated; see \cite{Bhar} and  \cite {Lissitsin}.

In our recent work \cite{DB1, DB2, DB3},  we considered the approximation
property and some of its variants for the predual of the weighted
spaces of holomorphic functions defined on open subsets of Banach
spaces. We continue this study in the present work. We aim at studying
some new characterizations of the bounded and compact approximation
property besides dealing with the $\mathcal{I}$-approximation property
for the predual of the weighted space of holomorphic mappings.

In Section 2, we give some basic notations, terminology and the results to be used in the sequel. In the next section we show that a Banach space $E$ has the compact approximation property if and only the inclusion map defined on a balanced open subset $U$ of $E$ can be approximated by a weighted holomorphic map with compact range.  Also, we  obtain some new characterizations of the bounded approximation property in terms of weighted spaces of holomorphic mappings.
Finally, in the last section, we prove the main result, namely - a
Banach space $E$ has the $\mathcal{I}$-approximation property if and
only if the predual $\mathcal{G}_w(U)$ has the
$\mathcal{I}$-approximation property for a suitably chosen operator ideal $\mathcal{I}$ and weight $\it{w}$ defined on a balanced open subset $U$ of a Banach space $E$.
\section{Preliminaries}
Throughout this paper, we shall use the letters $E$ and
$F$ to denote the complex Banach spaces and 
the symbol  $E^*$ to denote the topological
dual of $E$. We denote by $U$ a non-empty open subset of $E$ and by
$U_{E}$ the open unit ball of $E$.
The symbol $B^{\lambda}_
E$ denotes the closed ball of $E$ consisting of the elements with
norm $\leq \lambda$. For $\lambda=1$, $B^1_
E=B_E$ is the closed unit ball of $E$.
For  $m\in\mathbb{N}$, $\mathcal{P}(
^mE, F)$ denotes the Banach space of all continuous m-homogeneous
polynomials from $E$ to $F$.
A continuous polynomial $P$ is a
mapping from $E$ into $F$ which can be represented as a sum $P = P_0 + P_1 +\dots + P_l$
with $P_m \in \mathcal{P}
(^mE, F)$ for $m = 0, 1,\dots, l.$ The vector space of all continuous polynomials
from $E$ into $F$ is denoted by $\mathcal{P}(E, F)$. 
A  polynomial $P\in \mathcal{P}(^mE,F)$ is said to be of \emph{finite
	type} if it is of the form
$$
P(x)=\sum_{j=1}^{l}\phi_{j}^{m}(x)y_j,~x\in E, $$ where $\phi_j \in
E^*$ and $y_j \in F$, $1\leq j\leq l$. We denote by $
\mathcal{P}_f(^mE,F)$
the space of finite type
polynomials from $E$ into $F$.  A  polynomial $P\in \mathcal{P}(^mE,F)$ is said to be \emph{compact} if $P(B_E)$ is relatively compact in $F$. We denote by $\mathcal{P}_k(^mE,F)$ the space of all compact $m$-homogeneous polynomials. For $m=1$, 
$\mathcal{P}_k(^1E,F) \equiv \mathcal{K}(E; F)$ ($\mathcal{P}_f(^1E,F) \equiv \mathcal{F}(E; F)$) is the space of all compact (finite) linear operators
from $E$ to $F$.

A mapping $f:U\rightarrow F$ is said to be
\emph{holomorphic}, if for each $\xi\in U$, there exists a ball
$B(\xi, r)$ with center at $\xi$ and radius $r> 0$, contained in $U$
and a sequence $\{P_{m}\}_{m=0}^{\infty}$ of polynomials with $P_{m}
\in \mathcal{P}(^{m}E,F)$, $m\in \mathbb{N}_0$ such that
\begin{equation}
	f(x)=\sum_{m=0}^{\infty}P_{m}(x-\xi)= \sum_{m=0}^{\infty} \frac{1}{m!} \hat{d}^{m} f(\xi) (x-\xi),
\end{equation}
where the series converges uniformly for each $x\in B(\xi,r)$. The vector space of all holomorphic mappings from $U$ to $F$ is denoted by
$\mathcal{H}(U,F)$ and the compact open topology on $\mathcal{H}(U,F)$, the topology of uniform convergence on compact subsets of $U$, is denoted by  $\tau_{0}$. In case $U=E$,
the class $\mathcal{H}(E,F)$ is the space of entire mappings from
$E$ into $F$. For $F=\mathbb{C}$, we write $\mathcal{H}(U)$ for
$\mathcal{H}(U,\mathbb{C})$.

A subset $A$ of $U$ is said to be \emph{$U$-bounded} if $A$ is bounded
and there exists a neighborhood $V$ of $0$ such that $A+V\subset U$. A mapping $f$ in $\mathcal{H}(U,F)$ is of \emph{bounded
	type} if it maps $U$-bounded sets to bounded sets in $F$. The space of
holomorphic mappings of bounded type is denoted by
$\mathcal{H}_b(U,F)$. The space $\mathcal{H}_b(U,F)$ endowed with
the topology $\tau_b$, the topology of uniform convergence on
$U$-bounded sets, is a Fr\'{e}chet space, cf.\cite{Barroso}. We refer
to~\cite{Barroso, Dineen1, Dineen2, Mujica1, Nachbin}
for notations and background on infinite dimensional
holomorphy.

A \emph{weight} $\textit{w}$ on $U$ is a
continuous and strictly positive function satisfying
\begin{equation}
	0< \inf_{A}\textit{w}(x) \leq \sup_{A}\textit{w}(x)<\infty
\end{equation}
for each $U$-bounded set $A$. A weight $\textit{w}$ defined on an
open balanced subset $U$ of $E$ is said to be \emph{radial} if
$\textit{w}(tx)=\textit{w}(x)$ for all $x\in U$ and $t\in
\mathbb{C}$ with $|t|=1$; and in case of $U=E$ it is said to be \emph{rapidly
	decreasing} if $\sup\limits_{x\in E}{w(x)\|x\|^m}<\infty$ for each
$m\in \mathbb{N}_0$. Corresponding to a weight function \emph{w}, the weighted space of holomorphic
functions is defined as
$$\mathcal{H}_w(U;F)=\{f\in \mathcal{H}(U;F):\|f\|_w= \sup_{x\in
	U}\textit{w}(x)\|f(x)\|< \infty\}.$$ The space $(\mathcal{H}_w(U;F),
\|\cdot\|_w)$ is a Banach space and $B_w$ denotes its closed unit
ball. For $F=\mathbb{C}$, we write $\mathcal{H}_w(U)=
\mathcal{H}_w(U,\mathbb{C})$. It can be easily seen that the norm
topology $\tau_{\|\cdot\|_w}$ on $\mathcal{H}_w(U,F)$ is finer than
the topology induced by $\tau_0$, the topology of uniform convergence on compact subsets of $U$.
Since the closed unit ball $B_w$ of
$\mathcal{H}_w(U)$ is $\tau_0$-compact, the
predual of $\mathcal{H}_w(U)$ is given by
$$\mathcal{G}_w(U)=\{\phi \in \mathcal{H}_w(U)^\prime: \phi |B_w
~\textrm{is}~ \tau_0 -\textrm{continuous}\}
$$
by the Ng Theorem, cf.\cite{Ng}.\\
Let us quote the following notations and results from \cite{DB1}.
\begin{proposition}\label{poly containment in weighted holo}
	Let \it{w} be a weight on an open subset $U$ of a Banach space $E$. Then, for each $m \in\mathbb{N}$, the following are equivalent:
	
	(a) $\mathcal{P}(^mE, F)\subset \mathcal{H}_w(U, F)$ for each Banach space $F$.
	
	(b)  $\mathcal{P}(^mE)\subset \mathcal{H}_w(U)$.
\end{proposition}
\begin{theorem}(Linearization Theorem)\label{Linearization Theorem}
	For an open subset $U$ of a Banach space $E$ and a weight \it{w}
	on $U$, there exists a Banach space $\mathcal{G}_w(U)$ and a mapping
	$\Delta_{w}\in \mathcal{H}_w(U, \mathcal{G}_w(U))$ with $\|\Delta_w\|_w\leq 1$ satisfying the
	following property: for each Banach space $F$ and each mapping $f\in
	\mathcal{H}_w(U,F)$, there is a unique operator $T_f\in
	\mathcal{L}(\mathcal{G}_w(U), F)$ such that $T_f\circ\Delta_w=f$.
	The correspondence $\Psi$ between $\mathcal{H}_w(U,F)$ and
	$\mathcal{L}(\mathcal{G}_w(U), F)$ given by
	$$\Psi(f)=T_f$$
	is an isometric isomorphism. The space $\mathcal{G}_w(U)$ is
	uniquely determined upto an isometric isomorphism by these
	properties.
\end{theorem}
We write
$$\mathcal{H}_w(U) \otimes F=\{f\in \mathcal{H}_w(U,F):f ~\textrm{has 
	finite dimensional range}\}$$
and
$$
\mathcal{H}_w^c(U,F)= \{f\in \mathcal{H}_w(U,F):wf ~\textrm{has 
	relatively compact range}\}.
$$
The next proposition establishes the interplay between the
properties of a mapping $f\in \mathcal{H}_w(U,F)$ and  the
corresponding operator $T_f\in \mathcal{L}(\mathcal{G}_w(U),F)$.
\begin{proposition}\label{Interplay of holo and linear}
	Let $U$ be an open subset of a Banach space $E$ and $\it{w}$ be a weight on $U$. Then for any Banach space $F$, we have
	
	$(a)$ $f\in \mathcal{H}_w(U) \otimes F$ if and only if $T_f\in \mathcal{F}(\mathcal{G}_w(U),F)$.
	
	$(b)$ $f\in \mathcal{H}_w^c(U,F)$ if and only if $T_f\in \mathcal{K}(\mathcal{G}_w(U),F)$.
\end{proposition}

Further, we recall the locally convex topology $\tau_{\mathcal{M}}$ on
$\mathcal{H}_w(U,F)$  generated by the family $\{p_{\bar{\alpha},\bar{A}}: \bar{\alpha} = (\alpha_j) \in c^+_0
, ~\bar{A} = (A_j), A_j ~\textrm{being finite subset of U
	for each}~ j\}$ of semi-norms defined by
$$p_{\bar{\alpha},\bar{A}}( f ) = \sup_
{j\in \mathbb{N}}(\alpha_j\inf_{x\in A_j}w(x)\sup_{y\in A_j}\|f(y)\|),~~f\in \mathcal{H}_w(U,F).
$$
\begin{theorem}\label{topological isomorphism}
	Let $E$ and $F$ be Banach spaces and $\it{w}$ be a weight on an open subset of $E$. Then the mapping $\Psi:(\mathcal{H}_w(U,F), \tau_{\mathcal{M}})\rightarrow (\mathcal{L}(\mathcal{G}_w(U), F), \tau_c)$ is a topological isomorphism.
\end{theorem}
In case $\Psi$ is restricted to norm bounded subsets of $\mathcal{H}_w(U,F)$, we have the following result from \cite{DB3}.
\begin{theorem}\label{topo bounded iso}
	Let $E$ and $F$ be Banach spaces and $\it{w}$ be a weight on an open subset $U$ of $E$. Then the restriction of the map $\Psi:(\mathcal{H}_w(U,F), \tau_{c})\rightarrow (\mathcal{L}(\mathcal{G}_w(U), F), \tau_c)$ on $\|\cdot\|_w$-bounded subsets of $\mathcal{H}_w(U,F)$ is a topological isomorphism.
\end{theorem}

A Banach space $E$ is said to have the \emph{approximation property} (the comapct \emph{approximation property}) abbreviated as AP(CAP)
if for every compact set $K$ of $E$ and $\epsilon>0$, there exists
an operator $T\in {\mathcal{F}}(E,E)$ ($T\in {\mathcal{K}}(E,E)$) such that
$$\sup_{x\in K}\|T(x)-x\|<\epsilon.$$ If $T\in \mathcal{F}(E,E)$ can be chosen with $\|T\|\leq \lambda$ for some $\lambda$, $1\leq \lambda<\infty$, then $E$ is said to have the \emph{$\lambda$-bounded approximation property}($\lambda$-BAP). $E$ has the \emph{bounded approximation property}(BAP) if $E$ has  the $\lambda$-BAP for some $\lambda$, $1\leq \lambda<\infty$.

The following characterization of the bounded approximation property is due to Grothendieck and proved in \cite{Casazza}.
\begin{theorem}\label{bap char}
	For a Banach space $E$ and $1\leq \lambda < \infty$, the following are equivalent:\\
	(i) $E$ has the $\lambda$-bounded approximation property.\\ (ii)
	$\overline{B^{\lambda}_{\mathcal{F}(E,F)}}^{\tau_c}= B_{\mathcal{L}(E,F)}$ for every Banach space $F$.\\
	(ii)
	$\overline{B^{\lambda}_{\mathcal{F}(F,E)}}^{\tau_c}= B_{\mathcal{L}(F,E)}$ for every Banach space $F$.\\
	(iv)  $\overline{B^{\lambda}_{\mathcal{F}(E,E)}}^{\tau_c}= B_{\mathcal{L}(E,E)}$.
	
\end{theorem}

Similar to the  above characterization of the bounded approximation property, the following characterizations of the compact approximation property is quoted from \cite{Caliskan}, see also \cite{Casazza}.
\begin{theorem}\label{cap char}
	For a Banach space $E$, the following are equivalent:\\
	(i) $E$ has the compact approximation property.\\ (ii)
	$\overline{\mathcal{K}(E,F)}^{\tau_c}= \mathcal{L}(E,F)$ for every Banach space $F$.\\
	(iii) $\overline{\mathcal{K}(F,E)}^{\tau_c}= \mathcal{L}(F,E)$ for every Banach space $F$.
\end{theorem}
\begin{proposition}\label{cap poly char}
	For a Banach space $E$, the following are equivalent:\\
	(i) $E$ has the compact approximation property.\\ (ii)
	$\mathcal{P}(E,F)= \overline{\mathcal{P}_k(E,F)}^{\tau_c}$ for every Banach space $F$.
\end{proposition}
The next characterization of the compact approximation property for a Banach space $E$ in terms of the predual of the weighted spaces has been obtained in \cite{DB2}.
\begin{theorem}\label{cap holo char}
	Let $\it{w}$ be a be a radial weight on a balanced open subset $U$ of a Banach
	space $E$ such that $\mathcal{P}(^mE)\subset \mathcal{H}_w(E)$ for each $m\in \mathbb{N}$. Then, the following are equivalent:
	
	(i) $E$ has the compact approximation property.
	
	(ii) $\overline{\mathcal{P}_k(^mE, F)}^{\tau_\mathcal{M}}=\mathcal{H}_w(U,F)$ for each Banach space $F$ and for each $m\in \mathbb{N}$.
	
	(iii) $\overline{\mathcal{H}_w^c(U, F)}^{\tau_\mathcal{M}}=\mathcal{H}_w(U,F)$ for each Banach space $F$.
	
	(iv) $\mathcal{G}_w(U)$ has the compact approximation property.

\end{theorem}

\section{$\mathcal{H}_w(U,F)$ and the compact approximation property}
In this section, we study the compact approximation property for the space $\mathcal{G}_w(U)$. 
Let us begin with the following lemma.
\begin{lemma}\label{Lemma}
	For each $x\in U$, there exists $\epsilon>0$ and a $U$-bounded set  $V_{x,\epsilon}$ such that $x\in V_{x,\epsilon}\subset U$.
\end{lemma}
\begin{proof}
	Since $U$ is open, there exists $\epsilon>0$ such that $x+2\epsilon B_E\subset U$. Define $V_{x,\epsilon}=x+\epsilon B_E$. Then $V_{x,\epsilon}$ is $U$-bounded and $x\in V_{x,\epsilon}\subset U$.
\end{proof}
Let us recall from \cite{Aron} that a map $f\in \mathcal{H}(U, F)$ is said to be \emph{compact} if for each $x\in U$, there exists a neighborhood $V_x$ of $x$ such that $V_x\subset U $ and $f(V_x)$ is relatively compact in $F$. We denote by $\mathcal{H}_k(U,F)$ the space of all compact holomorphic mappings from $U$ to $F$. A characterization of compact holomorphic mappings in terms of Taylor series coefficients was otained in \cite{Aron} by R. Aron and M. Schottenloher. This was further generalized by E. Caliskan and P. Rueda \cite{CaRue} for $\xi$-balanced domains (an open set $U$ is said to be $\xi$-balanced if $(1-\mu)\xi+\mu x\in U$ for all $x\in U$ and $\mu\in \mathbb{C}$ with $|\mu|\leq 1$) in locally convex spaces of which a particular case is quoted below.
\begin{proposition}\label{Aron result}
	Let $U$ be a balanced open subset of a Banach space $E$ and $f\in \mathcal{H}(U,F)$. Then the folowing are equivalent:
	
	(a) $f\in \mathcal{H}_k(U,F)$.
	
	(b)  $P_m f (0) \in \mathcal{P}_k(^mE, F)$ for all $m\in\mathbb{N}$.
	
	(c)  $P_m f (x) \in \mathcal{P}_k(^mE, F)$ for all $m\in\mathbb{N}$ and $x\in U$.
	
\end{proposition}

Relating $
\mathcal{H}_w^c(U,F)$ with $\mathcal{H}_k(U,F)$, we prove the following inclusion result.

\begin{proposition}\label{Hw and HK}
	Let $U$ be an open subset of a Banach space $E$ and $\it{w}$ be a weight defined on $U$. Then $\mathcal{H}_w^c(U, F)\subset \mathcal{H}_k(U,F)$ for each Banach space $F$.
\end{proposition}
\begin{proof}
	Let $f\in \mathcal{H}_w^c(U,F)$. Then $T_f\in \mathcal{K}(\mathcal{G}_w(U), F$) by Proposition \ref{Interplay of holo and linear} (b). Fix $x\in U$ arbitraily. Then by Lemma \ref{Lemma}, there exists $\epsilon >0$ such that
	$$V_{x, \epsilon}=x+\epsilon B_E\subset U,$$
	where $V_{x, \epsilon}$ is $U$-bounded.
	By Theorem \ref{Linearization Theorem}, $f(V_{x, \epsilon})=T_f\circ \Delta_w(V_{x, \epsilon}$). Since $\Delta_w\in \mathcal{H}_w(U,\mathcal{G}_w(U)) \subset \mathcal{H}_b(U,\mathcal{G}_w(U))$,
	$\Delta_w(V_{x,\epsilon})$ is bounded in $\mathcal{G}_w(U)$. Thus $f(V_{x, \epsilon})$ is relativlely compact in $F$. Consequently $f\in \mathcal{H}_k(U,F)$.
\end{proof}
\begin{remark}
	Note that the reverse implication in Proposition \ref{Hw and HK} does not hold even in the particular case of $w\equiv 1$, cf. \cite[Example 3.2]{Mujica}.
\end{remark}
\begin{theorem}\label{CAP char}
	
	Let $\it{w}$ be a radial weight on a balanced open subset $U$ of a Banach space $E$
	such that $\mathcal{P}(E)\subset \mathcal{H}_w(U)$. Then the following assertions are equivalent:
	
	(a) $E$ has the CAP.
	
	(b) $\overline{\mathcal{H}_w^c(U,F)}^{\tau_{\mathcal{M}}}=\mathcal{H}_w(U,F)$ for every Banach space $F$.
	
	(c) $\mathcal{H}_w(U,F)\subset\overline{\mathcal{H}_w^c(U,F)}^{\tau_{c}}$  for every Banach space $F$.
	
\end{theorem}
\begin{proof}
	$(a) \implies (b)$ is the same as $(i)\Rightarrow (iii)$ of Theorem \ref{cap holo char}.

	$(b)\implies (c)$:  Since $\tau_c\leq \tau_{\mathcal{M}}$,  $\overline{\mathcal{H}_w^c(U,F)}^{\tau_{\mathcal{M}}}\subset  \overline{\mathcal{H}_w^c(U,F)}^{\tau_{c}}$ for each Banach space $F$. Thus the implication holds.
	
	$(c)\implies (a)$: Let $P\in \mathcal{P}(^mE,F)$. By hypothesis and Proposition \ref{poly containment in weighted holo},  $\mathcal{P}(E, F)\subset \mathcal{H}_w(U, F)$ for each Banach space $F$. Therefore $P\in \overline{\mathcal{H}_w^c(U,F)}^{\tau_{c}}$ by using $(c)$. Thus there exists a net $(f_{\alpha})_{\alpha\in \Lambda}\subset \mathcal{H}_w^c(U,F)$ such that $f_{\alpha}\xrightarrow{\tau_c} P.$ 
	
	For each $m\in\mathbb{N}$, define $Q_m:(\mathcal{H}(U,F), \tau_c)\rightarrow \mathcal{P}(^mE,F)$ as $Q_m(f)=P_mf(0)$, $f\in \mathcal{H}(U,F)$. Now, $Q_m$ is a continuous projection for each $m\in \mathbb{N}$, by \cite[Proposition 3.22]{Dineen2}. Therefore $Q_m(f_{\alpha})\xrightarrow{\tau_c} Q_m(P)$, that is, $P^mf_{\alpha}(0)\xrightarrow{\tau_c} P$, where $P^mf_{\alpha}(0)\in \mathcal{P}_k(^mE,F)$ as $f_{\alpha}\in  \mathcal{H}_w^c(U, F)$ and $\mathcal{H}_w^c(U, F)\subset\mathcal{H}_k(U,F)$. Thus 
	$\overline{\mathcal{P}_k(^mE,F)}^{\tau_c}=\mathcal{P}(^mE,F)$. Hence $E$ has the CAP by Proposition \ref{cap poly char}.
\end{proof}

\begin{theorem}\label{I_U cap char}
	
	Let $\it{w}$ be a radial weight on a balanced open subset $U$ of a Banach space $E$
	such that $\mathcal{P}(E)\subset \mathcal{H}_w(U)$. Then the following assertions are equivalent:
	
	(a) $E$ has the CAP.
	
	(b) $I_U\in \overline{\mathcal{H}_w^c(U,E)}^{\tau_{\mathcal{M}}}$, where $I_U:U\rightarrow E$ is the inclusion mapping.
\end{theorem}

\begin{proof}
	$(a)\implies (b)$:  Let $E$ has the CAP. Then by Theorem 3.5 ($(a)\implies (b)$),
	$\overline{\mathcal{H}_w^c(U,F)}^{\tau_{\mathcal{M}}}=\mathcal{H}_w(U,F)$. Since $I_U\in \mathcal{P}(^mE,E)$ and $\mathcal{P}(^mE,E)\subset \mathcal{H}_w(U,E)$, 
	$I_U\in \overline{\mathcal{H}_w^c(U,E)}^{\tau_{\mathcal{M}}}.$
	
	$(b)\implies (a)$: By (b), there exist a net $(f_{\alpha})_{\alpha\in \Lambda}\subset\mathcal{H}_w^c(U,E)$  such that $f_{\alpha}\xrightarrow{\tau_{\mathcal{M}}} I_U$. 
	By Theorem \ref{topological isomorphism}, $T_{f_{\alpha}}\xrightarrow{\tau_c}T_{I_U}\equiv T$. 
	
	Define $S=d^1\Delta_w(0)$. Then $S\in \mathcal{L}(E, \mathcal{G}_w(U))$ and, by the Cauchy integral formula $$S(t)=\frac{1}{2\pi i}\int_{|\xi|=r}\frac{\Delta_w(\xi t)}{\xi^2} \,d \xi, ~~t\in E,$$
	where $r>0$ is chosen such that $\{\xi t: |\xi|\leq r\}\subset U$. 
	Note that $T\circ \Delta_w=I_E$, where $I_E:E\rightarrow E$ is the identity operator on $E$. Therefore $$T\circ S(t)=\frac{1}{2\pi i}\int_{|\xi|=r}\frac{ t}{\xi} \,d \xi= t, ~\textrm{for all~} t\in E.$$
	Now, $T_{f_{\alpha}}\circ S\xrightarrow{\tau_c} T\circ S$ gives $T_{\alpha}\circ S\xrightarrow{\tau_c} I_E$. Since $T_{f_{\alpha}}\in \mathcal{K}(\mathcal{G}_w(U), E)$ by Proposition \ref{Interplay of holo and linear}(b), $T_{f_{\alpha}}\circ S\in \mathcal{K}(E, E)$ for each $\alpha\in \Lambda$ and $(a)$ follows.
\end{proof}
The next result of this section deals with a characterization of the bounded approximation property in terms of finite rank weighted holomorphic mappings.

\begin{theorem}
	Let $\it{w}$ be a bounded weight on the open unit ball $U_E$ of a Banach space $E$. Then the following assertions are equivalent:
	
	(a)	$E$ has the BAP.
	
	(b) $\overline{B_{\mathcal{H}_w(U_E)\bigotimes E}^{\lambda}}^{\tau_c}=B_{\mathcal{H}_w(U_E,E)}$ for some $\lambda$, $1\leq \lambda <\infty$.
	
	(c) $I_{U_E}\in \overline{B_{\mathcal{H}_w(U_E)\bigotimes E}^{{\lambda}^{\prime}}}^{\tau_c}$ for some $\lambda^{\prime}$, $1\leq \lambda^{\prime} <\infty$.
\end{theorem}
\begin{proof}
	$(a) \implies (b)$: Assume that $E$ has the $\lambda$-BAP  for some $\lambda$, $1\leq \lambda <\infty$ and $f\in B_{\mathcal{H}_w(U_E,E)}$. Then by Theorem \ref{Linearization Theorem}, $T_f\in B_{\mathcal{L}(\mathcal{G}_w(U_E), E)}$. Since $E$ has the $\lambda$-BAP, there exists a net $(T_{\alpha})_{\alpha\in \Lambda}\subset B^{\lambda}_{\mathcal{F}(\mathcal{G}_w(U_E), E)}$
	such that $T_{\alpha}\xrightarrow{\tau_c} T_f$ by Theorem \ref{bap char}($(i) \Leftrightarrow (iii)$).
	
	Define $f_{\alpha}= T_{\alpha}\circ \Delta_w$ for each $\alpha\in \Lambda$. Note that $(f_{\alpha})_{\alpha\in \Lambda}\subset \mathcal{H}_w(U_E)\bigotimes E$ by Proposition \ref{Interplay of holo and linear}(a) and $f_{\alpha} \xrightarrow{\tau_{c}} f$ by Theorem \ref{topo bounded iso}.  Also $\|f_{\alpha}\|_w\leq \|T_{\alpha}\|\leq \lambda$
	for each $\alpha \in \Lambda$ since $\|\Delta_w\|_w\leq 1$. Thus 
	$f\in\overline{B_{\mathcal{H}_w(U_E)\bigotimes E}^{\lambda}}^{\tau_c}$ and hence $(b)$ follows.
	
	$(b) \implies (c)$:	 Since $\it{w}$ is bounded on $U_E$, $I_{U_E}\in \mathcal{H}_w(U_E, E)$.  Now $\frac{I_{U_E}}{\|I_{U_E}\|_w}\in B_{\mathcal{H}_w(U_E, E)}$.
	
	By $(b)$, there exists a net  $(f_{\alpha})_{\alpha\in \Lambda}\subset \mathcal{H}_w(U_E)\bigotimes E$ with $\|f_{\alpha}\|\leq \lambda$ for each $\alpha\in \Lambda$, such that
	$f_{\alpha}\xrightarrow{\tau_{c}} \frac{I_{U_E}}{\|I_{U_E}\|_w}$. Define $g_{\alpha}=f_{\alpha}\|I_{U_E}\|_w$ for each $\alpha\in \Lambda$. Then $\|g_{\alpha}\|\leq \lambda \|I_{U_E}\|_w$ for each $\alpha\in \Lambda$ and $g_{\alpha}\xrightarrow{\tau_{c}} I_{U_E}$.
	Thus $ I_{U_E}\in \overline{B_{\mathcal{H}_w(U_E)\bigotimes E}^{\lambda^{\prime}}}^{\tau_c}$ with $\lambda^{\prime}=\lambda\|I_{U_E}\|_w$.
	
	$(c) \implies (a)$: Let $ I_{U_E}\in \overline{B_{\mathcal{H}_w(U_E)\bigotimes E}^{\lambda^{\prime}}}^{\tau_c}$. Then by Theorem \ref{topo bounded iso} there exists a net  $(f_{\alpha})_{\alpha\in \Lambda}\subset \mathcal{H}_w(U_E)\bigotimes E$ with $\|f_{\alpha}\|\leq \lambda^{\prime}$ for each $\alpha\in \Lambda$ such that $f_{\alpha}\xrightarrow{\tau_c}I_{U_E}$. Then $T_{f_{\alpha}}\xrightarrow{\tau_c}T_{I_{U_E}}$ and $\|T_{f_{\alpha}}\|= \|f_{\alpha}\|$ for each $\alpha\in \Lambda$. Proceeding as in the proof of Theorem 3.6((b)$\implies$(a)), we have 
	$T_{f_{\alpha}}\circ S\xrightarrow{\tau_c}T_{I_{U_E}}\circ S=I_E,$ where $S\in \mathcal{L}(E,\mathcal{G}_w(U_E))$ with $\| S\|\leq 1$. Thus $I_E\in \overline{B_{\mathcal{F}(E,E) }^{\lambda^{\prime}}}^{\tau_c}$ and hence $E$ has the BAP.
\end{proof}
\begin{remark}
	Let us note that the above result hold for the $\lambda$-BAP in case $w$ is a bounded weight defined on $U_E$ with $\sup\limits_{x\in U_E}w(x)\leq 1$. In case $w\equiv 1$, we have the following characterization of the $\lambda$-BAP for a Banach space $E$.
\end{remark}
\begin{proposition} Let $U_E$ be the  open unit ball of a Banach space $E$ and $1\leq \lambda < \infty$. Then for each Banach space $F$, the following assertions are equivalent:
	
	(a)	$E$ has the $\lambda$-BAP.
	
	(b) $\overline{B_{\mathcal{H}^{\infty}(V)\bigotimes E}^{\lambda}}^{\tau_c}=B_{\mathcal{H}^{\infty}(V,E)}$ for each open subset $V\subset F$.
	
	(c) $I_{U_E}\in \overline{B_{\mathcal{H}^{\infty}(U_E)\bigotimes E}^{{\lambda}}}^{\tau_c}$.	
\end{proposition}
\begin{proof}
	Since $\|P\|=1$ when $U=U_E$, cf.\cite[Proposition 2.3(b)]{Mujica}, the implications (a) $\implies$ (b) $\implies$ (c) are proved in \cite{Caliskan2} as (c) $\implies$ (d) $\implies$ (e) of Proposition 3.
	
	$(c) \implies (a)$ is a consequence of $(c)\implies (a)$ of Theorem 3.7 by taking $w\equiv 1$.
\end{proof}
\section{$\mathcal{H}_w(U,F)$ and  $\mathcal{I}-$ approximation property}

In this section we study  $\mathcal{I}-$ approximation property for the predual of a weighted space of holomorphic functions.
A Banach space $E$ is said to have \emph{$\mathcal{I}$-approximation property} if the identity operator on $E$ is uniformly
approximated by a member of $\mathcal{I}(E,E)$ on compact subsets of $E$.  

Let us first recall the multilinear/polynomial ideal generalization of an operator ideal from \cite{Botelho1, Botelho2}.
\begin{definition}\label{polynomial ideals} Let $\mathcal{I},~ \mathcal{I}_1, ~\mathcal{I}_2,\dots,\mathcal{I}_m$ be operator ideals.
	
	(a) A mapping $A\in \mathcal{L}(E_1, E_2, \dots, E_m;F)$ is said to be of \emph{type $\mathcal{L}[\mathcal{I}_1, \mathcal{I}_2, \dots, \mathcal{I}_m ]$} if there exist Banach spaces $G_1, G_2, \dots, G_m$, operators $u_j\in \mathcal{I}_j(E_j, G_j)$, $j=1,2,\dots, m$, and a mapping $B\in \mathcal{L}(G_1, G_2, \dots, G_m;F)$ such that $A=B\circ (u_1, u_2, \dots, u_m)$. For $\mathcal{I}=\mathcal{I}_1=\mathcal{I}_2=\dots=\mathcal{I}_m$, we write $\mathcal{L}[\mathcal{I}_1, \mathcal{I}_2, \dots, \mathcal{I}_m ]=\mathcal{L}[\mathcal{I}]$.
	
	(b) \emph{Composition Ideal of Multilinear Operators}:  A mapping $A\in \mathcal{L}(E_1, E_2, \dots, E_m;F)$ belongs to $\mathcal{I}\circ \mathcal{L}$ if there are Banach space $G$, a mapping $B\in\mathcal{L}(E_1, E_2, \dots, E_m;G)$ and $u\in \mathcal{I}(G,F)$ such that $A=u\circ B$.
	
	(c) \emph{Composition Polynomial Ideal}: A polynomial $P\in \mathcal{P}(^mE, F)$ belongs to $\mathcal{I}\circ \mathcal{P}$ if there exist a Banach space $G$ and $Q\in \mathcal{P}(^mG, F)$ and an operator $u\in \mathcal{I}(E,G)$ such that $P=u\circ Q$. In this case we write $P\in \mathcal{I}\circ \mathcal{P}(^mE, F)$.
\end{definition}
The following generalizations for $\mathcal{I}$-approximation property are quoted from \cite{Sonia}.
\begin{theorem}\label{I-AP characterization}
	For a Banach space $E$, the following are equivalent:\\
	(i) $E$ has the $\mathcal{I}$- approximation property.\\ (ii)
	$\overline{\mathcal{I}(E,F)}^{\tau_c}= \mathcal{L}(E,F)$ for every Banach space $F$.\\
	(iii) $\overline{\mathcal{I}(F,E)}^{\tau_c}= \mathcal{L}(F,E)$ for every Banach space $F$.
	
\end{theorem}
Observe that the above theorem is a generalization of Theorem \ref{cap char}. 
\begin{proposition}\label{comple I-AP} Let $\mathcal{I}$ be an operator ideal and $E$ be a Banach space with $\mathcal{I}$-AP. Then every complemented subspace of $E$ has $\mathcal{I}$-AP. 
\end{proposition}
The following characterization of $\mathcal{I}$-AP in terms of $m$-homogeneous polynomials is proved in \cite{Sonia}.
\begin{theorem}\label{I-AP polynomial ideal char}
	Let $E$ be a Banach space and $\mathcal{I}$ be an operator ideal such that $\mathcal{L}[I]\subset \mathcal{I}\circ \mathcal{L}$. Then  $E$ has $\mathcal{I}$-AP if and only if $\mathcal{P}(^mE, F)\subset \overline{\mathcal{I}\circ \mathcal{P}(^nE, F)}$ for every $n\in \mathbb{N}$ and Banach space $F$.
\end{theorem}

Let us define
$\mathcal{I}\circ \mathcal{H}_w(U,F)=\{f\in \mathcal{H}(U,F): f=S\circ g  ~~\textrm{for some Banach space }~G, ~S\in \mathcal{I}(G,F)~\textrm{and}~ g\in \mathcal{H}_w(U, G)\}.$
Generalizing Proposition \ref{Interplay of holo and linear} for an arbitrary operator ideal $\mathcal{I}$, we have 
\begin{theorem}\label{interplay wrt arb OI}
	Let $f\in \mathcal{H}_w(U,F)$ and $\mathcal{I}$ be an operator ideal. Then $f\in \mathcal{I}\circ \mathcal{H}_w(U,F)$ if and only if $T_f\in \mathcal{I}(\mathcal{G}_w(U), F)$.
\end{theorem}
\begin{proof}
	Suppose $f\in \mathcal{I}\circ \mathcal{H}_w(U,F)$. Then there exist a Banach space $G$, an operator $S\in \mathcal{I}(G, F)$ and a map $g\in \mathcal{H}_w(U,G)$ such that $f=S\circ g$.
	Note that $$T_f(\delta_x)=f(x)=S\circ g(x)=S \circ T_g(\delta_x).$$
	Since $span\{\Delta_w(x)=\delta_x:x\in U\}$ is dense in $\mathcal{G}_w(U)$, cf.\cite[Lemma 7]{Beltran}, ~$T_f=S \circ T_g\in \mathcal{I}(G_w(U), F)$.
	
	Conversely, assume that $T_f\in \mathcal{I}(\mathcal{G}_w(U), F)$ for  $f\in \mathcal{H}_w(U, F)$. Since $\Delta_w\in \mathcal{H}(U, \mathcal{G}_w(U))$ and $f=T_f\circ \Delta_w$, $f\in \mathcal{I}\circ \mathcal{H}_w(U, F)$.
\end{proof}

Next, we show that each mapping in $\mathcal{I}\circ \mathcal{P}$ admits a factorization in terms of weighted holomorphic mappings. 
\begin{proposition}\label{polynomials and holomorphic ideal}
	Let $\mathcal{I}$ be an operator ideal and $P:E\rightarrow F$ be a continuous polynomial such that $P=P_0+P_1+P_2+\cdots+P_n$ with $P_l\in \mathcal{I}\circ \mathcal{P}(^lE,F)$ for each $0\leq l\leq n$. Then $P\in \mathcal{I}\circ \mathcal{H}_w(U,F)$.
\end{proposition}
\begin{proof}
	Though the proof of this result is analogous to the  proof of \cite[Theorem 2.4 ]{ArBoPeRu}, we outline the same for the sake of convenience.
	
	Since $P=P_0+P_1+P_2+\cdots+P_n$ with $P_l\in \mathcal{I}\circ \mathcal{P}(^lE,F)$ for each $0\leq l\leq n$, there are Banach spaces $G_0$, $G_1$, $G_2$,\dots, $G_n$,  $Q_l\in \mathcal{P}(^lE,G_l)$ and $A_l\in \mathcal{I}(G_l,F)$ such that $P_l=A_l\circ Q_l$ for each $0\leq l\leq n$.
	
	Define $G=G_0\times G_1\times G_2\times\cdots\times G_n$, $Q:E\rightarrow G$ and $A:G\rightarrow F$ by
	$Q(x)=(Q_0(x), Q_1(x),\dots,Q_n(x))$
	and $A((y_0, y_1, y_2,\dots, y_n))=A_0(y_0)+A_1(y_1)+\cdots +A_n(y_n).$ Then $A\in \mathcal{I}(G, F)$ and $Q\in \mathcal{P}(E,G)$.
	
	Note that $A\circ Q(x)=A_0\circ Q_0(x) +A_1\circ Q_1(x)+A_2\circ Q_2(x)+\cdots+ A_n\circ Q_n(x)=P_0(x) +P_1(x)+P_2(x)+\cdots+ P_n(x)=P(x)$ for all $x\in E$, where $Q\in \mathcal{P}(E,G)\subset \mathcal{H}_w(U,G)$. Thus $P\in \mathcal{I}\circ \mathcal{H}_w(U,F)$.
\end{proof}
Using the above proposition, we obtain our main result. 
\begin{theorem}
	Let $w$ be a radial weight on a balanced open subset $U$ of a Banach space $E$ such that $\mathcal{H}_w(U, F)$ contains all polynomials and $\mathcal{I}$ be an operator ideal such that $\mathcal{L}[I]\subset \mathcal{I}\circ \mathcal{L}$. Then the following assertions are equivalent:
	
	(a) $E$ has $\mathcal{I}-$AP.
	
	(b) $\mathcal{P}(^mE, F)= \overline{\mathcal{I}\circ \mathcal{P}(^mE, F)}^{\tau_c}$ for each $m\in \mathbb{N}$ and each Banach space $F$.
	
	(c) $\mathcal{H}_w(U,F)=\overline{\mathcal{I}\circ \mathcal{H}_w(U, F)}^{\tau_{\mathcal{M}}}$ for each Banach space $F$.
	
	(d) $\mathcal{G}_w(U)$ has the $\mathcal{I}$-AP.
	
	(e) $I_U\in  \overline{\mathcal{I}\circ \mathcal{H}_w(U, F)}^{\tau_{\mathcal{M}}}$ for each Banach space $F$.
\end{theorem}
\begin{proof}
	$(a)\Leftrightarrow (b)$ follows from \cite[Theorem 4.2$(a) \Leftrightarrow (f)$]{Sonia}.
	
	$(b)\implies (c)$: 
	Let $p$ be a $\tau_{\mathcal{M}}$- continuous semi-norm on $\mathcal{H}_w(U,F)$ and $f\in \mathcal{H}_w(U,F)$. Since $\mathcal{P}(E,F)$ is $\tau_{\mathcal{M}}$-dense in $\mathcal{H}_w(U,F)$, cf. \cite[Proposition 4.6]{DB1}, there exists $P\in \mathcal{P}(E,F)$ such that
	\begin{equation}\label{fP}
		p(f-P)<\frac{\epsilon}{2}.
	\end{equation}
	
	Let $P=P_0+P_1+\cdots+P_l$, $P_m\in\mathcal{P}(^mE,F)$ for each $0\leq m\leq l$. Then by (b), there exists $Q_m\in \mathcal{I}\circ \mathcal{P}(^mE,F)$  such that
	\begin{equation}\label{QP}
		p(Q_m-P_m)<\frac{\epsilon}{2(m+1)}
	\end{equation}
	for each $0\leq m\leq l$.
	
	Define $Q=Q_0+Q_1+\cdots+Q_l$. By Proposition \ref {polynomials and holomorphic ideal}, $Q\in  \mathcal{I}\circ \mathcal{H}_w(U,F)$  and
	\begin{equation}\label {PQ}
		p(P-Q)= p(\sum_{m=0}^{l}P_m-\sum_{m=0}^{l}Q_m)<\frac{\epsilon}{2}
	\end{equation}
	by (\ref{QP}).
	Hence by (\ref{fP}) and (\ref{PQ}), we have
	$$p(Q-f)<\epsilon.$$
	Thus (c) follows.
	
	$(c)\implies (d)$: Since $\Delta_w\in \mathcal{H}_w(U, \mathcal{G}_w(U))$, $\Delta_w\in  \overline{\mathcal{I}\circ \mathcal{H}_w(U, \mathcal{G}_w(U))}^{\tau_{\mathcal{M}}}$ by taking $F=\mathcal{G}_w(U)$ in $(c)$. Thus there exists a net
	$(f_{\alpha})\subset \mathcal{I}\circ \mathcal{H}_w(U, \mathcal{G}_w(U))$ such that $f_{\alpha}\xrightarrow{\tau_{\mathcal{M}}}\Delta_w$. By Theorem \ref{topological isomorphism}, we get
	$$T_{f_{\alpha}}\xrightarrow{\tau_{c}}T_{\Delta_w}.$$
	Note that $T_{\Delta_w}=I_{\mathcal{G}_w(U)}$ and $T_{f_{\alpha}}\in \mathcal{I}(\mathcal{G}_w(U), \mathcal{G}_w(U))$ by Theorem \ref{interplay wrt arb OI}. Thus
	$I_{\mathcal{G}_w(U)}\in \overline{\mathcal{I}(\mathcal{G}_w(U), \mathcal{G}_w(U))}^{\tau_c}$ and hence $(d)$ follows.
	
	$(d)\implies (a)$: Since $E$ is complemented in $\mathcal{G}_w(U)$, cf. \cite[ Proposition 3.6]{DB1}, $(a)$ follows from $(d)$ and Proposition \ref{comple I-AP}.
	
	$(c)\implies (e)$: Since $I_U\in \mathcal{P}(^mE, F)\subset \mathcal{H}_w(U,F)$, $(e)$ follows from $(c)$.
	
	$(e) \implies (a)$: As $I_U\in  \overline{\mathcal{I}\circ \mathcal{H}_w(U, E)}^{\tau_{\mathcal{M}}}$, there exists a net $(f_{\alpha})\subset \mathcal{I}\circ \mathcal{H}_w(U, E)$ such that $f_{\alpha}\xrightarrow{\tau_{\mathcal{M}}}I_U$.
	Also, by Theorem \ref{topological isomorphism}, $T_{f_{\alpha}}\xrightarrow{\tau_c} T_{I_U}$ . Then proceeding as in the proof of Theorem \ref{I_U cap char}, $T_{f_{\alpha}}\circ S\xrightarrow{\tau_c}I_E$, where $S\in \mathcal{L}(E, \mathcal{G}_w(U))$. As $T_{f_{\alpha}}\in \mathcal{I}(\mathcal{G}_w(U), E)$ for each $\alpha$ in view of Theorem \ref{interplay wrt arb OI}, $E$ has the $\mathcal{I}-$AP by Theorem \ref{I-AP characterization}.
\end{proof}

\section*{Acknowledgement}
The first author would like to thank BITS-Pilani for the financial support received (Grant No. RIG/2019/0641).

\end{document}